\documentclass{article}

\usepackage{amssymb}
\usepackage{latexsym}
\usepackage{amsmath}
\usepackage{amsthm}
\usepackage{amscd}

\def\mycite#1#2{\cite{#1}(#2)}

\newfont{\cyr}{wncyr10}
\newfont{\cyb}{wncyb10}

\begin{document}

\setlength{\textwidth}{8in}%
\setlength{\topmargin}{-.5in}
\setlength{\textheight}{9in}

\def\cc {{\mathfrak c}}
\def\ii {{\mathfrak i}}
\def\UU {{\mathfrak U}}
\def\CC {{\Bbb C}}
\def\HH {{\Bbb H}}
\def\NN {{\Bbb N}}
\def\PP {{\Bbb P}}
\def\QQ {{\Bbb Q}}
\def\RR {{\Bbb R}}
\def\TT {{\Bbb T}}
\def\ZZ {{\Bbb Z}}
\def\sA {{\mathcal A}}
\def\sB {{\mathcal B}}
\def\sC {{\mathcal C}}
\def\sD {{\mathcal D}}
\def\sE {{\mathcal E}}
\def\sF {{\mathcal F}}
\def\sG {{\mathcal G}}
\def\sH {{\mathcal H}}
\def\sI {{\mathcal I}}
\def\sJ {{\mathcal J}}
\def\sK {{\mathcal K}}
\def\sL {{\mathcal L}}
\def\sM {{\mathcal M}}
\def\sN {{\mathcal N}}
\def\sO {{\mathcal O}}
\def\sP {{\mathcal P}}
\def\sQ {{\mathcal Q}}
\def\sR {{\mathcal R}}
\def\sS {{\mathcal S}}
\def\sT {{\mathcal T}}
\def\sU {{\mathcal U}}
\def\sV {{\mathcal V}}
\def\sW {{\mathcal W}}
\def\sX {{\mathcal X}}
\def\sY {{\mathcal Y}}
\def\sZ {{\mathcal Z}}
\def\od {\mathrm{od}}
\def\cf {\mathrm{cf}}
\def\dom {\mathrm{dom}}
\def\id {\mathrm{id}}
\def\int {\mathrm{int}}
\def\cl {\mathrm{cl}}
\def\Hom {\mathrm{Hom}}
\def\ker {\mathrm{ker}}
\def\log {\mathrm{log}}
\def\nwd {\mathrm{nwd}}

\def\TT{{\Bbb T}}
\def\Z{{\Bbb Z}}
\def\ZZ{{\Bbb Z}}

\newcommand{\s }{\mathcal }
\newcommand{\bA}{\mathbf{A}}
\newcommand{\bB}{\mathbf{B}}
\newcommand{\bC}{\mathbf{C}}
\newcommand{\bD}{\mathbf{D}}
\newcommand{\bI}{\mathbf{I}}
\newcommand{\bE}{\mathbf{E}}
\newcommand{\bK}{\mathbf{K}}
\newcommand{\bT}{\mathbf{T}}

\newtheorem{thm}{Theorem}[section]

\newtheorem{theorem}[thm]{Theorem}
\newtheorem{corollary}[thm]{Corollary}
\newtheorem{lemma}[thm]{Lemma}
\newtheorem{claim}[thm]{Claim}
\newtheorem{axiom}[thm]{Axiom}
\newtheorem{conjecture}[thm]{Conjecture}
\newtheorem{conventions}[thm]{Conventions}
\newtheorem{fact}[thm]{Fact}
\newtheorem{hypothesis}[thm]{Hypothesis}
\newtheorem{assumption}[thm]{Assumption}
\newtheorem{proposition}[thm]{Proposition}
\newtheorem{criterion}[thm]{Criterion}
\newtheorem{definition}[thm]{Definition}
\newtheorem{definitions}[thm]{Definitions}
\newtheorem{discussion}[thm]{Discussion}
\newtheorem{example}[thm]{Example}
\newtheorem{notation}[thm]{Notation}
\newtheorem{remark}[thm]{Remark}
\newtheorem{remarks}[thm]{Remarks}
\newtheorem{problem}[thm]{Problem}
\newtheorem{terminology}[thm]{Terminology}
\newtheorem{question}[thm]{Question}
\newtheorem{questions}[thm]{Questions}
\newtheorem{notation-definition}[thm]{Notation and Definitions}
\newtheorem{acknowledgement}[thm]{Acknowledgement}

\title{Some Classes of Minimally Almost Periodic Topological groups
\footnote{This paper derives from and extends selected portions of the Doctoral
Dissertation \cite{gouldphd}, written at Wesleyan University (Middletown, Connecticut, USA)
by the second-listed co-author under the guidance of the first-listed co-author.}}

\author{W.W. Comfort\footnote{Department of Mathematics and Computer Science,
 Wesleyan University, Wesleyan Station,
Middletown, CT 06459; phone: 860-685-2632; FAX 860-685-2571
Email: wcomfort@wesleyan.edu}, 
Franklin R. Gould\footnote{Department of Mathematical
Sciences, Central Connecticut State University, New Britain, 06050;
Email: gouldfrr@ccsu.edu}}

\maketitle

\begin{abstract}
A topological group $G=(G,\sT)$ has the {\it small subgroup generating
property}
(briefly: {\it has the SSGP property}, or
is {\it an SSGP group}) if for each neighborhood $U$ of
$1_G$ there is a family $\sH\subseteq\sP(U)$
of subgroups of $G$ such that
$\langle\bigcup\sH\rangle$
is dense in $G$. The class of $\rm{SSGP}$ groups is defined and 
investigated with respect
to the properties usually studied by topologists (products,
quotients, passage to dense subgroups, and the like), and with respect to
the familiar class of minimally almost 
periodic groups (the m.a.p. groups). Additional classes 
SSGP$(n)$ for $n<\omega$ (with 
SSGP$(1) = {\rm SSGP}$) are defined and investigated, and
the class-theoretic inclusions

SSGP$(n)\subseteq{\rm SSGP}(n+1)\subseteq{\rm m.a.p.}$\\
\noindent are established and shown proper.

In passing the authors also establish the
presence of SSGP$(1)$ or SSGP$(2)$ in many of 
the early examples in the literature of abelian m.a.p. groups.
\end{abstract}



\maketitle

\section{Introduction}
\begin{conventions}
{\rm
(a) For $X$ a set we write $\sP(X):=\{A:A\subseteq X\}$. This is {\it the
power set} of $X$.

(b) The topological spaces we hypothesize, in particular our
hypothesized topological groups, are
assumed to be completely regular and Hausdorff ({\it i.e.}, to be
Tychonoff spaces). When a topology is defined or constructed
on a set or a group, the Tychonoff property will be verified
explicitly 
(if it is not obvious). In this context we recall (\mycite{hri}{8.4}) that
in order that a topology with continuous algebraic operatons on a group
be a Tychonoff topology, it suffices that it satisfy the Hausdorff
separation property.

(c) For $X$ a space and $x\in X$ we write

$\sN_X(x):=\{U\subseteq X: U$ is a neighborhood of $x\}$.\\
\noindent When ambiguity is unlikely we write $\sN(x)$ in place of $\sN_X(x)$.

(d) The identity of a group $G$ is denoted $1$ or $1_G$; if $G$ is known
or assumed to be abelian and additive notation is in play, the identity
may be denoted $0$ or $0_G$.

(e) When $G$ is a group and $\kappa\geq\omega$, we use the notations $\bigoplus_\kappa\,G$ and $G^{(\kappa)}$ interchangeably:

$\bigoplus_\kappa\,G=G^{(\kappa)}:=
\{x\in G^\kappa:|\{\eta<\kappa:x_\eta\neq1_G\}|<\omega\}$.\\
\noindent When $G$ is a topological group, $\bigoplus_\kappa\,G$ has the topology inherited from $G^\kappa$.
}
\end{conventions}

The {\it minimally almost periodic}
groups (briefly: the m.a.p. groups)
to which our title refers are by definition those topological
groups $G$
for which every continuous homomorphism $\phi:G\rightarrow K$
with $K$ a compact group satisfies $\phi[G]=\{1_K\}$.
It follows from the Gel$'$fand-Ra\u{\i}kov Theorem \cite{gelrai} (see
\cite{hri}(\S22) for a detailed development and proof) that every compact
group $K$ is algebraically and topologically a subgroup of a group of the
form $\Pi_{i\in I}\,U_i$ with each $U_i$ a (finite-dimensional) unitary
group~\cite{hri}(22.14). Therefore, to check that a topological group $G$
is m.a.p. it suffices to
show that each continuous homomorphism $\phi:G\rightarrow U(n)$ with
$U(n)$ the $n$-dimensional unitary group satisfies
$\phi[G]=\{1_{U_n}\}$. Similarly, since every compact abelian group $K$ is
algebraically and topologically isomorphic to a subgroup of a
group of the form
$\TT^I$~\cite{hri}(22.17), to check that an abelian
topological group $G$ is m.a.p., it suffices to
show that each continuous homomorphism $\phi:G\rightarrow\TT$ satisfies
$\phi[G]=\{1_{\TT}\}$. 

Sometimes for convenience we denote by m.a.p. the (proper) class of
m.a.p. groups, and if $G$ is a m.a.p.
group we write $G\in$ {\rm m.a.p.}. Similar conventions apply to the classes SSGP$(n)$
($0\leq n<\omega$) defined in Definition~\ref{defSSGPn}.

Algebraic characterizations of those abelian groups which admit an m.a.p. group
topology has been achieved only recently~\cite{diksha14}. For a brief
historical account of
the literature touching this issue, see Discussion~\ref{history} below.

\begin{acknowledgement}
{\rm
We gratefully acknowledge helpful comments received from
Dieter Remus, 
Dikran Dikranjan, and Saak Gabriyelyan. Each of them improved the
exposition in a pre-publication version of this
manuscript, and enhanced our historical commentary with additional
bibliographic references.
}
\end{acknowledgement}

\section{SSGP Groups: Some Generalities}

\begin{definition}\label{topgen}
{\rm
Let $G=(G,\sT)$ be a topological group and let $A\subseteq G$. Then
$A$ {\it topologically generates} $G$ if
$\langle A\rangle$ is dense in $G$.
}
\end{definition}

\begin{definition}\label{defSSGP}
{\rm
Let $G=(G,\sT)$ be a topological group. Then
$G$ has the {\it small subgroup generating property} if for every
$U\in\sN(1_G)$ there is a family $\sH$ of subgroups of $G$
such that $\sH\subseteq\sP(U)$ and $\bigcup\sH$
topologically generates $G$.

}
\end{definition}
A topological group with the small subgroup generating property is said
to {\it have the {\rm SSGP} property}, or to {\it be an {\rm SSGP} group}, or
simply to {\it be} SSGP.

Now for $0\leq n<\omega$ the classes SSGP$(n)$
are defined as follows.

\begin{definition}\label{defSSGPn}
{\rm
Let $G=(G,\sT)$ be a Hausdorff topological group. Then

(a) $G\in{\rm SSGP}(0)$ if $G$ is the trivial group.

(b) $G\in{\rm SSGP}(n+1)$ for $n \ge 0$ if for every $U\in\sN(1_G)$ 
there is a family $\sH$ of subgroups of $G$ such that \\
\hspace*{1cm} (1) $\sH\subseteq\sP(U)$, \\
\hspace*{1cm} (2) $H:=\overline{\langle \bigcup\sH \rangle}$ is normal in $G$, and\\
\hspace*{1cm} (3) $G/H\in{\rm SSGP}(n)$.
}
\end{definition}

\begin{remarks}\label{inclusions}
{\rm
(a) For $0\leq n<\omega$ the class-theoretic inclusion
${\rm SSGP}(n)\subseteq{\rm SSGP}(n+1)$ holds, hence
${\rm SSGP}(n)\subseteq{\rm SSGP}(m)$ when $n<m<\omega$. To see this,
note that when
$G\in{\rm SSGP}(n)$ and $U\in\sN(1_G)$ then we have, taking
$\sH:=\{\{1_G\}\}$, that $H:=\overline{\langle\bigcup\sH\rangle}=\{1_G\}$
and $G/H\simeq G\in{\rm SSGP}(n)$, so indeed $G\in{\rm SSGP}(n+1)$.

(b) Clearly the class SSGP of Definition~\ref{defSSGP}
coincides with the class SSGP$(1)$ of Definition~\ref{defSSGPn}.
}
\end{remarks}

A topological group $G$ is said to be {\it precompact} if $G$ is a
(dense) topological subgroup of a compact group. It is a theorem of
Weil~\cite{weili} that a topological group $G$ is precompact if and
only if $G$ is {\it totally bounded} in the sense that for each
$U\in\sN(1_G)$ there is finite $F\subseteq G$ such that $G=FU$. 

It is obvious that a precompact group $G$ with $|G|>1$ is not m.a.p. Indeed
if $G$ is dense in the compact group $\overline{G}$ then the continuous function
$\id:G\hookrightarrow\overline{G}$ does not satisfy $\id[G]=\{1_{\overline{G}}\}$.

\begin{theorem}\label{SSGPvsmap}
The class-theoretic inclusion ${\rm SSGP}(n)\subseteq{\rm m.a.p.}$
holds for each $n<\omega$.
\end{theorem}
\begin{proof}
The proof is by induction on $n$.
Clearly if $G\in{\rm SSGP}(0)$ and $\phi\in\Hom(G,U(m))$
then $\phi[G]=\{1_{U(m)}\}$, so $G\in{\rm m.a.p.}$
Suppose now that SSGP$(n)\subseteq {\rm m.a.p.}$, let $G$ be a
topological group such that $G\in$ SSGP$(n+1)$, and let
$\phi:G\rightarrow U(m)$ be a continuous homomorphism.
Choose $V \in \sN(1_{U(m)})$ so that $V$ contains
no subgroups of $U(m)$ other than $\{1_{U(m)}\}$.  Then
$U:= \phi^{-1}[V]\in\sN(1_G)$, and $\phi$ maps every subgroup of $U$ to 
$1_{U(m)}$.   Let $\sH \subseteq \sP(U)$ be a family of 
subgroups of $G$ such that $H:=\overline{\langle\sH\rangle}$ is 
normal in $G$ and such that $G/H\in$ SSGP$(n)$.  
Since a homomorphism maps subgroups to subgroups we have 
$\phi[H] = \{1_{U(m)}\}$.  It follows that $\phi$ defines a continuous 
homomorphism $\widetilde{\phi}: G/H \rightarrow U(m)$ (given by
$\widetilde{\phi}(xH):=\phi(x)$).  By the induction 
hypothesis, $\widetilde{\phi}$ is the trivial homomorphism, so $\phi$ is 
trivial as well; the relation $G\in{\rm m.a.p.}$ follows.
\end{proof}

Now in \ref{opensubnotSSGP}--\ref{admit} we clarify what is and 
is not known about the classes of groups mentioned in 
Theorem~\ref{SSGPvsmap}.  We begin with a simple lemma 
and a  familiar definition.

\begin{lemma}\label{opensubnotSSGP}
Let $G$ be a nontrivial (Hausdorff) topological group for which
some $U \in \sN(1_G)$ contains no subgroup other than 
$\{1_G\}$. Then there is no $n<\omega$ such that $G\in\,${\rm SSGP}$(n)$.
\end{lemma}

\begin{proof}
Clearly $G\notin{\rm SSGP}(0)$.  Suppose there is a 
minimal $n>0$ such that $G\in{\rm SSGP}(n)$, and let
$U\in\sN(1_G)$  be as
hypothesized. Then the only choice for $\s H \subseteq \sP(U)$ is 
$\sH:=\{\{1_G\}\}$, yielding $H = \overline{ \langle \cup \sH \rangle} = \{1_G\}$.  
Thus, $G/H = G \in$ SSGP$(n-1)$, which contradicts the 
assumption that $n$ is minimal.
\end{proof}

\begin{definition}\label{defcogen}
{\rm
Let $G$ be a group and let $1_G\notin C\subseteq G$. Then $C$ {\it
cogenerates} $G$ if every subgroup $H$ of $G$ such that $|H|>1$
satisfies $H\cap C\neq\emptyset$.
}
\end{definition}

\begin{theorem}\label{finitelycogennotSSGP}
Let $G$ be a nontrivial finitely cogenerated topological group.
Then there is no
$n<\omega$ such that $G\in$ {\rm SSGP}$(n)$.
\end{theorem}
\begin{proof}
Let $C$ be a finite set of cogenerators for $G$, and choose 
$U \in \sN(1_G)$ such that $U \bigcap C = \emptyset$.  Then $U$ 
contains no subgroup other than $\{1_G\}$, and the statement follows
from Lemma~\ref{opensubnotSSGP}.
\end{proof}

We have noted for every $n<\omega$ the class-theoretic inclusion 
SSGP$(n)\subseteq$ m.a.p.  On the other hand, there are many 
examples of $G\in$ m.a.p. such that $G\in{\rm SSGP}(n)$ for no 
$n<\omega$.  But more is true: There are groups 
which admit an m.a.p. topology which admit an SSGP$(n)$ 
topology for no $n<\omega$.
Indeed from Corollary~\ref{ZnoSSGP} and
Theorem~\ref{finitelycogennotSSGP} respectively
we see that the groups $G=\ZZ$ and
$G=\ZZ(p^\infty)$ (cogenerated by
suitable $C\subseteq G$
with $|C|=p-1<\omega$) admit no
${\rm SSGP}(n)$ topology; while 
Ajtai, Havas, and Koml\'{o}s~\cite{ajtai}, and later 
Zelenyuk and Protasov~\cite{zelenyuk}, have shown the existence of m.a.p. topologies
for $\ZZ$ and for $\ZZ(p^\infty)$.

In Theorem \ref{dont_admit} we show that in the context of abelian groups,
Theorem~\ref{finitelycogennotSSGP} can be strengthened.  We 
use the following basic facts from the theory of abelian groups.  

\begin{lemma}\label{finitelygen1} 
A finitely cogenerated group is the direct sum of finite cyclic 
p-groups and groups of the form $\ZZ({p^\infty})$, hence is torsion
{\rm (\cite{fuchsi}(3.1 and 25.1).} 
\end{lemma}

\begin{lemma}\label{finitelygen2} 
A finitely generated abelian group is the direct sum 
of cyclic free groups and cyclic torsion groups
{\rm(\cite{fuchsi}(15.5))}.
\end{lemma}

\begin{lemma}\label{finitelygen3}
If $G$ is a finitely generated abelian group and $H$ is a 
torsionfree subgroup then there is a decomposition $G = K \oplus T$ where $T$ 
is the torsion subgroup, $K$ is torsionfree and $H \subseteq K$  
{\rm (\cite{fuchsi}, Chapter III)}.
\end{lemma}

\begin{theorem}
\label{dont_admit}
A nontrivial abelian group which is the direct sum of a finitely generated group and a finitely
cogenerated group does not admit an ${\rm SSGP}(n)$ topology for any $n<\omega$.  
\end{theorem}

\begin{proof}  We proceed by induction on the torsionfree rank, $r_0(G)$.  Suppose first that
$r_0(G)=0$.  Then $G$ is finitely co-generated and does not admit an ${\rm SSGP}(n)$ 
topology by Theorem
\ref{finitelycogennotSSGP}.  Now suppose that the theorem has been proved up to rank $r-1$ 
and we have $r_0(G) = r \ge 1$ and $G = F \oplus T$, with F finitely generated and T finitely 
co-generated.  Using Lemmas \ref{finitelygen1} and \ref{finitelygen2}, 
we rewrite $G$ in the form 
$G = F^\prime \bigoplus T^\prime$ where
$T^\prime$ is the 
(finitely cogenerated) torsion subgroup and
$F^\prime$ is free.
Then $r_0(F^\prime) = r_0(G) = r$.   Let
$a \in F^\prime$ be an element of infinite order
and choose $U \in \s{N}(0)$ so that
$a \notin \overline{U}$ and so that
$\overline{U} \bigcap C = \emptyset$ where $C$
is a finite set of
cogenerators of $T^\prime$ (with $0 \notin C$).
If all subgroups contained in $U$ are torsion,
then each such subgroup is a subgroup of $T^\prime$ and is therefore the 
zero subgroup, since it misses $C$.  In that case, $G$ does not have SSGP.   
Alternatively, if $U$ contains a cyclic subgroup
$H$ of infinite order, we have $r_0(\overline{H}) > 0$.  Furthermore, since
$\overline{H} \subseteq \overline{U}$, we have $\overline{H} \bigcap T^\prime = \{0\}$.  
It follows from Lemma \ref{finitelygen3} that there is a decomposition 
$G = F^{\prime \prime} \bigoplus T^\prime$ 
which is isomorphic to the original decomposition and is
such that $\overline{H} \subseteq F^{\prime \prime}$.  Since a quotient of a finitely 
generated group is also
finitely generated, it follows that $F^{\prime \prime} / \overline{H}$ is finitely generated.  
Then we have $G / \overline{H} = (F^{\prime \prime} / \overline{H} ) \bigoplus T^\prime$.  
We also have that $r_0(G / \overline{H}) < r$ because 
$r_0(G) = r_0(\overline{H}) + r_0(G / \overline{H})$,  
(\cite{fuchsi}(\S16, Ex.~3(d))).    
Also, $G / \overline{H}$ is nontrivial since 
$\overline{H} \subseteq \overline{U}$ and
$a \notin \overline{U}$.  It follows by our induction assumption that 
$G / \overline{H}$ does not admit ${\rm SSGP}(n)$, and so by 
Theorem~\ref{prsrv_prop}(b) (below), neither does $G$.
\end{proof}

\begin{corollary}\label{ZnoSSGP}
The group $\ZZ$ does not admit an SSGP$(n)$
topology for any $n<\omega$.
\end{corollary}

The following theorem lists several inheritance properties 
for groups in the classes SSGP$(n)$.
 
\begin{theorem}\label{prsrv_prop}
\rm{(a)} If $K$ is a closed normal 
subgroup of $G$, with $K \in$ SSGP$(n)$ 
and $G/K \in$ SSGP$(m)$ then $G \in$ SSGP$(m + n)$.

{\rm (b)} If $G \in$ SSGP$(n)$ and $\pi: G \twoheadrightarrow B$ is a continuous 
homomorphism from $G$ onto $B$, then 
$B \in$ SSGP$(n)$.  In particular, if $K$ is a 
closed normal subgroup of $G\in{\rm SSGP}(n)$ then $G/K \in$ SSGP$(n)$.

{\rm (c)} If $K$ is a dense subgroup of $G$ and
$K \in$ SSGP$(n)$ then 
$G \in$ SSGP$(n)$.

{\rm (d)} If $G_i \in$ SSGP$(n)$ for each $i \in I$ then 
$\bigoplus_{i \in I} G_i \in$ SSGP$(n)$ and
$\prod_{i \in I} G_i \in$ SSGP$(n)$.
\end{theorem}

\begin{proof}
We proceed in each case by induction on $n$. Each statement is trivial when $n=0$. We address (a), (b), (c) and (d) in order, assuming in each case for $1\leq n<\omega$ that the statement holds for $n-1$.

(a) Let $U \in \s N(1_G)$, so that 
$U \cap K \in \s N(1_K)$.  Then there is a family
$\s H \subseteq \sP(U \cap K)$ of subgroups
of $K$ such that $K/H \in$ SSGP$(n-1)$ 
where $H:= \overline{\langle \cup \s H \rangle}^K = 
\overline{\langle \cup \s H \rangle}^G$.  Since
$G/K$ is topologically isomorphic with $(G/H)/(K/H)$,
we have 
$(G/H)/(K/H) \in$ SSGP$(m)$ along with $K/H \in$ SSGP$(n-1)$.  
Then by the induction hypothesis, $G/H \in$ SSGP$(m + n-1)$.  
Since $\s H \subseteq \s P(U)$ with $U$ 
arbitrary, we have $G \in$ SSGP$(m + n)$,
as required.

{\rm (b)} Given $G \in$ SSGP$(n)$ and
continuous $\pi: G \twoheadrightarrow B$,
let $U \in \s N (1_B)$.  Then $\pi^{-1}[U] \in \s N (1_G)$ and there is a 
family $\s H \subseteq \s P (\pi^{-1}[U])$
of subgroups such that 
$G / \overline{\langle \cup \sH \rangle} \in $ SSGP$(n-1)$.  Let $\widetilde{\s H}$ 
be the family of subgroups of $B$ given by $\widetilde{\s H} := \{\pi[L] : L \in \s H \}$.  
Then $\widetilde{\s H} \subseteq \s P (U)$.  
Set $H := \overline{\langle \cup \s H \rangle }$ and set
$\widetilde{H} := \overline{\langle \cup \widetilde{\s H} \rangle}$.  Then $\widetilde{H}$ 
is normal in $B$ since by assumption $H$ is normal in $G$.  
By invoking the induction hypothesis we will show that 
$B/ \widetilde{H} \in$ SSGP$(n-1)$ and thus that
$B \in$ SSGP$(n)$.  
Note that $H \subseteq \pi^{-1}[\widetilde{H}]$ since 
$\langle \cup \s H \rangle \subseteq \pi^{-1} [\langle \cup \widetilde{\s H} \rangle]$ 
and $\pi^{-1}[\widetilde{H}]$ is closed.  We have that
$G/H \in$ SSGP$(n-1)$ so 
by induction,
$(G/H) / \left ( \pi^{-1}[\widetilde{H}]/H \right )
\in$ SSGP$(n-1)$
and this is topologically isomorphic with $G /  \pi^{-1}[\widetilde{H}]$ by the 
second topological isomorphism theorem.  Now, we claim that the algebraic 
isomorphism $\widetilde{\pi}: G /  \pi^{-1}[\widetilde{H}] \rightarrow B/\widetilde{H}$ induced by $\pi$
is continuous (though it may not be open).  Clearly, $\pi$ maps cosets of 
$\pi^{-1}[\widetilde{H}]$ to cosets of $\widetilde{H}$.  If $\widetilde{V}$ is an open union of 
cosets of $\widetilde{H}$, then $\pi^{-1}[\widetilde{V}]$ is an open union of cosets of 
$\pi^{-1}[\widetilde{H}]$ and the claim follows.
Composing maps gives 
a continuous homomorphism from $G/H$ onto
$B/\widetilde{H}$; we 
conclude that $B/ \widetilde{H} \in$ SSGP$(n-1)$
and thus $B \in$ SSGP$(n)$, as required.

\rm{(c)} Given $G$ and $K$ as hypothesized, 
let $U \in \sN (1_G)$.  Since $U \cap K \in \sN(1_K)$, 
there is a family
$\s H \subseteq \sP(U \cap K)$ of subgroups
of $K$ such that 
$K/H  \in$ SSGP$(n-1)$, where
$H :=  \overline{\langle\bigcup\sH \rangle}^K$.
Note that $H = \overline{H}^G \cap K$.  Let 
$\phi: K \overline{H}/ \overline{H} \rightarrow K/(K \cap \overline{H})$ 
be the natural isomorphism from the first (algebraic) isomorphism theorem
for groups.  The corresponding theorem for topological groups says that 
$\phi$ is an open map, i.e., $\phi^{-1}$ is a continuous map.  
Then from part (a) of this theorem, 
$K\overline{H}/\overline{H} \in \rm{SSGP}(n-1)$.  Now
$K\overline{H}/\overline{H} $ is dense in $G/\overline{H}$,
because the subset of $G$ that projects onto the closure of 
$K\overline{H}/\overline{H} $ must be closed and must contain 
$K\overline{H}$.  Then
$G/\overline{H} \in$ SSGP$(n-1)$
by the induction hypothesis. Since 
$\overline{H}^G = \overline{\langle\bigcup\sH \rangle}^G$
we have $G \in$ SSGP$(n)$, as required.  

\rm{(d)}
Since $\bigoplus_{i\in I}\,G_i$ is dense in
$\Pi_{i\in I}\, G_i$, it suffices by part (c) to treat the case
$G:=\bigoplus_{i\in I}\,G_i$. Let $U\in\sN_G(1_G)$,
without loss of generality
with $U = \prod_{i\in I} U_i$ 
where $U_i\in \s N(1_{G_i})$ 
and $U_i = G_i$ for $i > N_U$.  For each $i$
there is a family 
$\s H_i \subseteq \s P(U_i)$ of subgroups of $G_i$
such that $G_i / H_i \in$ SSGP$(n-1)$ where 
$H_i := \overline{\langle \cup \, \s H_i \rangle}$.  Now consider the family of 
subgroups of $G$ given by
$\s H := \{\bigoplus_{i \in I} L_i: L_i \in \s H_i \}$.  
Then $ \s H \subseteq \s P(U)$, $\langle \cup \s H \rangle$ 
is identical to $ \bigoplus_{i \in I} \langle \cup \s H_i \rangle$,
and
$H := \overline{\langle \cup \s H \rangle}$
is identical to $\bigoplus_{i \in I} H_i$.  
We also have that $G / H$ is topologically isomorphic with 
$\bigoplus_{i \in I} G_i /H_i$ (cf. \cite{hri}(6.9)).  
From the induction hypothesis we have
$G/H \in$ SSGP$(n-1)$, so 
$G \in \rm{SSGP}(n)$, as required.
\end{proof}

\begin{remark}
\rm
{
Certain other tempting statements of inheritance
or permanence type,
parallel in spirit to those considered in
Theorem~\ref{prsrv_prop}, do not hold in general. We give some
examples.

(a) We show below, using a construction of Hartman and Mycielski~\cite{hartmyciel58}
and of Dierolf and Warken~\cite{Dierolf-Warken}, that
a closed subgroup of an SSGP group 
may lack the SSGP$(n)$ property for every $n<\omega$.
Indeed, every topological 
group can be embedded as a closed subgroup of an SSGP group
(Theorem~\ref{HMDW}).

(b) The conclusion of part (a) of Theorem~\ref{prsrv_prop}
can fail when $s<m+n$ 
replaces $m+n$ in its statement.  For example, the 
construction used in Lemma \ref{Gn_induct} shows that
a topological group  
$G \notin$ SSGP$(n)$ may have a closed normal subgroup
$K \in$ SSGP$(1)$ 
with also $G/K \in$  SSGP$(n)$, so $s=m+n$ is minimal when $m=1$.  We did not pursue the issue of minimality
of $m+n$ in Theorem~\ref{prsrv_prop}(a)
for arbitrary $m,n>1$.

(c) The converse to Theorem~\ref{prsrv_prop}(c)
can fail.  In \cite{gouldphd} 
a certain monothetic m.a.p. group constructed by Glasner 
\cite{glasner} is shown to have SSGP, but we noted above in
Corollary~\ref{ZnoSSGP} that $\ZZ$ admits an
SSGP$(n)$ topology for no $n<\omega$.

In contrast to that phenomenon, it should be
mentioned that (as has been noted by many authors)
in the
context of m.a.p. groups,
a dense subgroup $H$ of a topological group $G$
satisfies $H\in$ m.a.p. if and only if $G\in$ m.a.p.
Thus in particular in the case of Glasner's monothetic
group, necessarily the dense subgroup $\ZZ$
inherits an m.a.p. topology.
}
\end{remark}

We now restrict our discussion to abelian groups and to 
the class SSGP$ = $SSGP$(1)$, and
examine which specific abelian groups do and do not admit an SSGP
topology.  We have already noted (Theorem \ref{dont_admit}) that the 
product of a finitely cogenerated abelian group with a finitely generated 
abelian group does not admit an SSGP topology even though it may 
admit an m.a.p. topology.  We now give additional examples
of abelian groups 
which admit not only an m.a.p. topology but also an  SSGP topology.

\begin{theorem}\label{admit} The following abelian groups admit 
an {\rm SSGP} topology.

{\rm (a)} $\QQ$, and those subgroups of $\QQ$ in which some primes are 
excluded from denominators, as long as an infinite number 
of primes and their powers are allowed;

{\rm (b)} $\QQ/ \ZZ$ and $\QQ^\prime / \ZZ$ where $\QQ^\prime$ is 
a subgroup of $\QQ$ as in described in {\rm (a)};

{\rm (c)} direct sums of the form $\bigoplus_{i < \omega} \ZZ_{p_i}$ 
where the primes $p_{i}$ all coincide or all differ;

{\rm (d)} $\ZZ^{(\omega)}$ (the direct sum);

{\rm (e)} $\ZZ^\omega$ (the full product);

{\rm (f)} $G^{(\lambda)}$ for $|G|>1$ and $\lambda\geq\omega$;

{\rm (g)} $F^\lambda$ for $1<|F|<\omega$ and
$\lambda\geq\omega$;

{\rm (h)} arbitrary sums and products of groups which admit an 
{\rm SSGP} topology.
\end{theorem}

Item (h) is a special case of Theorem~\ref{prsrv_prop}(d),
and item (e) is 
demonstrated in the second author's paper \cite{gould14}.
The ``coincide'' case of item (c) follows from item (f), the
``differ'' case  is estabished below in Theorem~\ref{cyclic_sum_diff}.
Theorem~\ref{DW_sub}((c) and (d)) 
below demonstrates the validity of item (f) for
$\omega\leq\lambda\leq\cc$. This
together with (h) gives 
(d) and (f) in full generality.
Item (g) then follows from the relation
$F^\lambda\simeq\bigoplus_{2^\lambda}\,F$
(\cite{fuchsi}(8.4, 8.5)).
The remaining items are demonstrated 
in \cite{gouldphd}.

There are many 
examples of nontrivial SSGP$(1)$ groups (that is, of SSGP groups).
It has been shown by Hartman and Mycielski~\cite{hartmyciel58}
that every topological group G embeds as a closed subgroup into a
connected, arcwise connected group $G^*$; two decades later Dierolf and
Warken~\cite{Dierolf-Warken}, working independently and without
reference to \cite{hartmyciel58}, found essentially the same embedding
$G\subseteq G^*$ and showed that $G^*\in{\rm m.a.p.}$.
Indeed the arguments of \cite{Dierolf-Warken} show in effect
that $G^*\in{\rm SSGP}$
(of course with property SSGP not yet having been named).
We now describe the construction and we give briefly the relevant argument.

\begin{definition}\label{deirolf-warken}
{\rm
Let $G$ be a Hausdorff topological group.  Then algebraically
$G^*$ is the group of step functions $f: [0,1) \rightarrow G$ with 
finitely many steps, each of the form
$[a,b)$ with $0 \le a < b \le 1$.  
The group operation is pointwise multiplication in $G$.  The
topology $\sT$ on $G^*$ is  
the topology generated by (basic) 
neighborhoods of the identity function $1_{G^*}\in G^*$ of the form
 
$N(U, \epsilon) := \{f \in G^*: \lambda(\{x\in[0,1):f(x)\notin U\})<
\epsilon\}$,\\
\noindent where $\epsilon > 0$, $U \in \sN_G(1_G)$, and $\lambda$ 
denotes the usual Lebesgue measure on $[0,1)$.
}
\end{definition}

\begin{theorem}\label{HMDW}
Let $G$ be a topological group. Then

{\rm (a)} $G$ is closed in $G^*=(G^*,\sT)$;

{\rm (b)} $G^*$ is arcwise connected; and

{\rm (c)} $G^*\in~{\rm SSGP}$.
\end{theorem}

\begin{proof}
Note first that the association of each $x\in G$ with the function $x^*\in G^*$
(the function given by $x^*(r):=x$ for all $r\in[0,1)$) 
realizes $G$ algebraically as a subgroup of $G^*$. Furthermore
the map $x\rightarrow x^*$ is a homeomorphism onto its range,
since for $\epsilon < 1$, $U\in\sN(1_G)$ and $x\in G$ one has

$x\in U\Leftrightarrow x^*\in N(U,\epsilon)$.

{\rm(a)} Let $f_0 \in G^*$ and $f_0 \notin G$.
There are distinct (disjoint)
subintervals of $[0,1)$ on which $f_0$ assumes distinct values $g_0,g_1\in G$
respectively.  By the 
Hausdorff property there is $U \in \sN(1_G)$ such that 
$g_1  U \cap g_2 U = \emptyset$.  Choose $\epsilon$
smaller than the measure of either of the two
indicated intervals. Then $ f_0 N(U, \epsilon)$ is a 
neighborhood of $f_0$ such that $f_0N(U,\epsilon)\cap G=\emptyset$. Therefore, 
$G$ is closed in $G^*$.

{\rm (b)} Let $f \in G^*$ and for each $t \in [0,1)$ define 
$f_t: [0,1) \rightarrow G$ by $f_t(x) = f(x)$ for $0 \le x < t$ and $f_t(x) = 1_G$ 
for $t \le x < 1$; and define $f_1:=f$.  Then $t \mapsto f_t$ is a 
continuous map from $[0,1]$ to $G^*$ such that $f_0 = 1_{G^*}$ and $f_1 = f$.  
To show that the map is continuous, let $f_t  N(U, \epsilon)$ be a 
basic neighborhood of $f_t$ and let $s \in (t - \epsilon/4, \; t + \epsilon/4) \cap [0,1]$.  
Then $f_s \in f_t N(U, \epsilon)$, since

$\lambda(\{x\in[0,1):f_s(x)-f_t(x)\in U\})<\epsilon$.

\noindent We conclude that 
$G^*$ is arcwise connected.

{\rm (c)} Let $N(U,\epsilon)\in\sN(1_{G^*})$, and
for each interval $I=[t_0,t_1)\subseteq[0,1)$
with $t_1-t_0<\epsilon$ let

$F(I):=\{f\in G^*:f{\rm ~is~constant~on~}I, f\equiv1_G{\rm~on~}[0,1)\backslash I\}$.

\noindent Then $F(I)$ is a subgroup of $G^*$
and $F(I)\subseteq N(U,\epsilon)$, and with $\sH_\epsilon:=\{F(I)\}$ we have
that each $f\in G^*$ is 
the product of finitely many elements from $\bigcup\sH_\epsilon$---i.e.,
$f\in\langle\bigcup\sH_\epsilon\rangle\subseteq \overline{\langle\bigcup\sH_\epsilon\rangle}$.

It follows that $G^* \in{\rm SSGP}\subseteq{\rm m.a.p.}$.
\end{proof}

There are also countable subgroups of $G^*$ which 
retain properties (a) and (c) (but not (b)) of Theorem~\ref{HMDW}.  We make the 
following definition.

\begin{definition}\label{deirolf-warken_sub}
{\rm
Let $G$ be a topological group and let
$A\subseteq [0,1)$
where $A$ is dense in $[0,1)$ and $0 \in A$.  Then
$G^*_A = (G^*_A, \s T)$ is the subgroup of $(G^*, \s T)$ 
obtained by restriction of step functions on $[0,1)$ to 
those steps $[a,b)$ such that $a, b \in A \cup \{1\}, \; a < b$.
}
\end{definition}

\begin{theorem}\label{DW_sub}
Let $G$ be a topological group. Then

{\rm (a)} $G$ is closed in $G^*_A =(G^*_A,\sT)$; 

{\rm (b)} $G^*_A$ is dense in $G^*$;

{\rm (c)} $G^*_A \in~{\rm SSGP}$; and

{\rm (d)} if $G$ is abelian, then the groups
$G^*_A$, $G^{(\lambda)}$ (with $\lambda=|A|$)
are isomorphic as groups.
\end{theorem}

\begin{proof}
With the obvious required change, the proofs of
(a) and (c) coincide with the corresponding 
proofs in Theorem \ref{HMDW}.  

{\rm (b)} Let $f \in G^*$ have  $n$ steps ($n<\omega$)
and let $f\cdot N(U,\epsilon)\in\sN_{G^*}(f)$. Then there is
$\widetilde{f}\in f\cdot N(U,\epsilon)\cap G^*_A$ such
that $\widetilde{f}$ has
step end-points in $A \cup \{1\}$, each within 
$\epsilon /n$ of the corresponding end-point for $f$.

{\rm (d)}  We give an explicit isomorphism.  $G^{(\lambda)}$ can 
be expressed as the set of functions $\phi: A \rightarrow G$ with
finite support and pointwise addition.  Each such function is
the sum of finitely many elements of the form
$\phi_{a,g}$ with $a \in A$, 
$g \in G$, $\phi_{a,g}(a) = g$ and $\phi_{a,g}(x) = 0$ for $x \ne a$.  
Now we define corresponding functions $f_{a,g} \in G^*_A$.  
Let $f_{0,g}(x) = g$ for all $x \in [0,1)$ and for $a > 0$,
let $f_{a,g}$ be the two-step function defined by 
$f_{a,g}(x) = g$ for $0 \le x < a$ and $f_{a,g}(x) = 0$ for 
$a \le x < 1$.  Then the map $\phi_{a,g} \mapsto f_{a,g}$ 
extends linearly to an isomorphism from $G^{(\lambda)}$ 
onto $G^*_A$.
\end{proof}

\begin{remark}  {\rm Note that Theorem \ref{prsrv_prop}(c)
cannot be used to prove (c) from (b) in
Theorem~\ref{DW_sub}.  Note also that the isomorphism 
given in the proof of (d) provides a way of imposing an 
SSGP topology on $G^{(\lambda)}$ for $\omega \le \lambda \le c$ 
and $G$ an abelian group.  A corresponding 
mapping can be given when $G$ is nonabelian but it 
need not be an isomorphism.  
In that case, it is still possible to write each element of 
$G^*_A$ as a product of two-step functions, but then 
it is necessary to specify the order in which they are 
to be multiplied.}
\end{remark}

Some other SSGP groups arise as a consequence
of the following fact.

\begin{theorem}\label{torsionSSGP} 
Let $G=(G,\sT)$ be a (possibly nonabelian) torsion
group of bounded order
such that $(G,\sT)$ has no proper open subgroup.
Then $G\in{\rm SSGP}$.
\end{theorem}

\begin{proof}
There is an integer M which bounds the order of each
$x\in G$, and then $N:=M!$ satisfies $x^N=1_G $ for each $x\in G$.

We must show: Each $U\in\sN(1_G)$ contains a family $\sH$ of subgroups 
such that $\langle\bigcup\sH\rangle$ is dense in $G$. Given such $U$,
let $V\in\sN(1_G)$ 
satisfy $V^N\subseteq U$. For each 
$x\in V$ we have $x^k\in U$ for $0\leq k\leq N$, hence
$x\in V\Rightarrow\langle x\rangle\subseteq U$.  Thus with
$\sH:=\{\langle x\rangle:x\in V\}$ we 
have: $\sH$ is a family of subgroups of $U$ (that is, of subsets of
$U$ which are subgroups of $G$). Then $V\subseteq\bigcup\sH$, so
$G=\langle V\rangle\subseteq\langle\bigcup\sH\rangle$---the first equality 
because $\langle V\rangle$ is an open subgroup of $G$.
\end{proof}

In Corollaries \ref{connectedSSGP} and \ref{3equivs}
we record
two consequences of Theorem~\ref{torsionSSGP}.

\begin{corollary}\label{connectedSSGP}
If $(G,\sT)$ is a (possibly nonabelian) connected
torsion group of bounded order, then
$(G,\sT)\in{\rm SSGP}$.
\end{corollary}

\begin{proof}
A connected group has no proper open subgroup, so
Theorem~\ref{torsionSSGP} applies.
\end{proof}

\begin{lemma}\label{mapnopensub}
Let $G\in{\rm m.a.p.}$  and $G$ abelian.  
Then $G$ does not contain
a proper open subgroup.
\end{lemma}
\begin{proof}
Suppose that $H$ is a proper open subgroup
of $G$. Since $G/H$ is a nontrivial abelian 
discrete (and therefore locally compact) group, 
there is a nontrivial
(continuous) homomorphism $\phi:G/H\rightarrow \TT$.
Then the composition of $\phi$ with the projection map from 
$G$ to $G/H$ is a nontrivial continuous homomorphism 
from $G$ to a compact group, contradicting the m.a.p. property of $G$.
\end{proof}

\begin{remark}
{\rm
We are grateful to Dikran Dikranjan for the helpful reminder
that Lemma~\ref{mapnopensub} fails
when the ``abelian'' hypothesis is omitted.
Examples to this effect abound, samples including: (a)~the infinite algebraically simple
groups whose only group topology is the discrete topology, as
concocted by Shelah~\cite{shelah} under [CH], and by Hesse~\cite{hessei}
and Ol$'$shanski\u{\i}~\cite{olshai} (and later by several others) in [ZFC];
and (b)~such matrix groups as $SL(2,\CC)$, shown by von
Neumann~\cite{neumann} to be m.a.p. even in the discrete topology
(the later treatments \cite{neumannwig}, \cite{hri}(22.22(h))
and \cite{comfi}(9.11) of this specific group follow closely those of \cite{neumann}). 
}
\end{remark}

\begin{corollary}\label{3equivs} 
For an abelian torsion group $G$
of bounded order,  these
conditions are equivalent for each group
topology $\sT$ on $G$.

{\rm (a)} $(G,\sT)\in{\rm SSGP}$;

{\rm (b)} $(G,\sT)\in{\rm m.a.p.}$; and

{\rm (c)} $(G,\sT)$ has no proper open normal subgroup.
\end{corollary}

\begin{proof}
The implications (a)~$\Rightarrow$~(b),
(b)~$\Rightarrow$~(c),
and (c)~$\Rightarrow$~(a) are given respectively
by Theorem~\ref{SSGPvsmap},
Lemma~\ref{mapnopensub}, and Theorem \ref{torsionSSGP}.
\end{proof}

\begin{remark}\label{MGR}
{\rm
It is worthwhile to note that connected torsion groups of
bounded order, as hypothesized in Theorem~\ref{connectedSSGP},
do exist. 
For the reader's convenience, drawing freely on the expositions
\cite{smitht74}
and \cite{comfi}(2.3--2.4), we outline the essentials
for $0<n<\omega$
of a construction, first given by Markov~\cite{mark}, \cite{markii} and
Graev~\cite{graevi}, 
of a nontrivial abelian connected torsion group $G$ of order $n$. Let
$X$ be a Tychonoff space and let

$G:=\{\Sigma_{i=1}^N\,k_ix_i:k_i\in\ZZ, N<\omega, x_i\in X\}$\\
\noindent be the free abelian
topological group on the alphabet
$X$ with $0_G=0$, and for continuous $f:X\rightarrow H$ with
$H$ a topological abelian group
define $\overline{f}:G\rightarrow H$
by $\overline{f}(\Sigma_{i=1}^N\,k_ix_i)=
\Sigma_{i=1}^N\,k_if(x_i)\in H$. It is
easily checked, as in the sources cited, that (a)~in
the (smallest) topology $\sT$ making each
such $\overline{f}$ continuous, $(G,\sT)$ is
a (Hausdorff) topological group; (b)~the map $x\rightarrow1\cdot x$ from $X$
to $G$ maps $X$ homeomorphically onto a closed
topological subgroup of $G$; and (c)~$G$ is
connected if (and only if) $X$ is connected.

Now take $X$ compact connected and fix $n$ such that
$0<n<\omega$.
It suffices to show that (1)~$nX$
is a proper closed subset of $G$, and (2)~ every proper closed
subset $F\subseteq X$ generates a proper closed subgroup $\langle F\rangle$ of $(G,\sT)$; for then the
group $G/\langle nX\rangle$ will be as desired, since
$a\in G\Rightarrow na\in\langle nX\rangle$.

(1)~$nX$ is compact in $G$, hence closed. Define
$f_0:X\rightarrow\RR$ by $f_0\equiv1$; then
$\overline{f_0}\equiv n$ on $nX$, while for $x\in X$ we have
$\overline{f_0}((n+1)x)=n+1$, so $(n+1)x\notin nX$.

(2) Given $x\in X\backslash F$ choose continuous $f_1:X\rightarrow\RR$  such that $f_1(x)=1$,
$f_1\equiv0$ on $F$.
Then $\overline{f_1}\equiv0$ on $\langle F\rangle$ and $\overline{f_1}(x)=1$, so
$x=1\cdot x\notin\langle F\rangle$; so $\langle F\rangle$ is proper in $G$. If $a=\Sigma_{i=1}^N\,k_ix_i\in G\backslash\langle F\rangle$ there is $i_0$ such
that $x_{i_0}\notin F$, and with continuous
$f_2:X\rightarrow\RR$ such that $f_2(x_{i_0})=1$,
$f_2(x_i)=0$ for $i\neq i_0$ and $f_2\equiv0$ on $F$
we have $\overline{f_2}(a)=k_{i_0}$
and $\overline{f_2}\equiv0$
on $\langle F\rangle$. Then
$U:=\overline{f_2}^{-1}(k_{i_0}-1/3,k_{i_0}+1/3)\in\sN_G(a)$ and
$U\cap\langle F\rangle=\emptyset$; so $\langle F\rangle$ is closed in $G$.

We note that as in Corollary~\ref{connectedSSGP} the group
$G$, being connected, has no  proper open subgroup.
}
\end{remark}

\section{SSGP Groups: Some Specifics}

The question naturally arises whether for $n<\omega$ the class-theoretic
inclusion
SSGP$(n)\subseteq{\rm SSGP}(n+1)$ is proper. The issue is addressed in
part by
Prodanov~\cite{prod80} (of course, the classes SSGP$(n)$ had not been
formally defined in 1980); he
provided on the direct sum $G:=\ZZ^{(\omega)}=\bigoplus_\omega\,\ZZ$ a topological group
topology $\sT$ which satisfies (as we show below)
$(G,\sT)\in~{\rm SSGP}(2)$ and
$(G,\sT)\notin~{\rm SSGP}(1)$.
Given $G$, Prodanov~\cite{prod80} constructs a basis at 
$0$ for a group topology $\sT$ as follows:  Let $e_m \; (m = 1, 2, \ldots)$ 
be the canonical basis for $G$. Then, use 
induction to define a sequence of finite subsets of $\ZZ^{(\omega)}$:
\begin{quote}
``Let $A_1 = \{e_1 - e_2,\, e_2\}$, and suppose that the sets 
$A_1,\, A_2,\, \ldots ,\, A_{m-1}$ \;\; $(m = 2, 3, \ldots )$ are already defined.  By $\alpha_m$ we 
denote an integer so large that the $s$-th co-ordinates of all elements of $A_1 \cup A_2 \cup 
\ldots \cup A_{m-1}$ are zero for $s \geq \alpha_m$.  Now we define $A_m$ to consist of all 
differences
\begin{quote} 
$(1)$ $\displaystyle{e_{i + k \alpha_{m}} - e_{i + (k+1) \alpha_m} \hspace{.5cm}
(1 \leq i \leq m, \; 0 \leq k \leq 2^{m-1} -1)}$ 
\end{quote} 
and of the elements 
\begin{quote}
$(2)$ $\displaystyle{e_{i + 2^{m - 1} \alpha_m} \hspace{2.2cm} (1 \le i \le m)}$ .
\end{quote}
Thus the sequence $\{A_m\}^\infty_{m=1}$ is defined.\\
Now for arbitrary  $n \ge 1$ we define
\begin{quote}
$(3)$ $U_n:= (n+1)! \ZZ^{(\omega)} \pm A_n \, \pm 2 A_{n+1} \, \pm \ldots 
\pm 2^l A_{n + l} \pm \ldots.$''
\end{quote}
\end{quote}
(By the notation of (3) Prodanov means that $U_n$ consists of those elements
of $\ZZ^{(\omega)}$ 
which can be represented as a finite sum consisting of an element
divisible by $(n+1)!$ plus 
at most one element of $A_n$ with arbitrary sign, plus at most two
elements of $A_{n+1}$ with 
arbitrary signs, plus at most four elements of $A_{n+2}$
with arbitrary signs, and so on.)
\begin{quote}
``It follows directly from that definition that the sets $U_n$ are
symmetric with respect to $0$, 
and that $U_{n+1} + U_{n+1} \subset U_n \; (n = 1, 2, \ldots)$.
Therefore they form a 
fundamental system of neighborhoods of $0$ for a group topology
$\sT$ on $\ZZ^{(\omega)}$."
\end{quote}

Since we need it later, we give a careful proof of an additional fact
outlined only briefly by Prodanov~\cite{prod80}.

\begin{theorem}\label{prodT2}
The group $\ZZ^{(\omega)}$ with the group topology $\sT$ defined above is 
Hausdorff.
\end{theorem}

\begin{proof}
It suffices to show $\bigcap_{n<\omega}\,U_n=\{0\}$.

Let $0\neq g = \sum_{ i } a_i e_i$ where the $e_i$ form the canonical basis.  
Then there is a least integer $r$ such that $a_i = 0$ for $i \ge r+1$, and there is 
a least integer $s$ such that $(s+1)!$ does not divide $g$.  Let $n := \max(r, s)$.   We claim that if  $n > 0$ then $g \notin U_n$.  Suppose otherwise.  Since $(n+1)!$ does not divide $g$, 
there is some $p \leq n$ 
such that $(n+1)!$ does not divide $a_p$.  Thus any representation of $g$ in the form $(3)$ must 
include, for at least one $m > n$, one or more terms of the form
$\pm(e_{p} - e_{p +  \alpha_m}) $, all with the same sign.   This means that 
the components $\pm e_{p +  \alpha_m}$ must be cancelled by the same components from 
additional terms of the form 
$\pm(e_{p + \alpha_m} - e_{p + 2  \alpha_m}) $.  This chain of implications continues until 
$k$ reaches its maximum value, $2^{m-1} -1$, with the inclusion of terms 
$\pm(e_{p + (2^{m-1} -1)\alpha_m} - e_{p + 2^{m-1}  \alpha_m}) $.  Finally, the components 
$\pm e_{p + 2^{m-1} \alpha_m}$ must be cancelled by terms of type $(2)$ with $i=p$ and the  
same value of $m$.    This means that we have necessarily included at least 
$2^{m-1}+1$ elements from $A_m$ in our expansion of $g$, contradicting the requirement that 
no more than $2^{m-n}$ elements of $A_m$ be included as summands for such representations 
of $g \in U_n$. 
\end{proof}

\begin{theorem}\label{prodSSGP2not1}
Prodanov's group $(G,\sT)$ satisfies $(G,\sT)\in{\rm SSGP}(2)$,
 $(G,\sT)\notin{\rm SSGP}(1)$.
\end{theorem}
\begin{proof}
First we show that $(\ZZ^{(\omega)},\sT)\in{\rm SSGP}$$(2)$.  Since 
the $U_n$ form a basic set of neighborhoods of $0$, every neighborhood of 
$0$ contains a subgroup of the form $(n+1)! \ZZ^{(\omega)}$ for some $n$.
Thus each $U_n$
generates $\ZZ^{(\omega)}$.  This means that $Z^{(\omega)}$ has no
proper open subgroups, hence for fixed $n$ the group
$G_n := \ZZ^{(\omega)} / \overline{(n+1)! \ZZ^{(\omega)}}$
contains no proper open subgroup. Further,
$G_n$ is of bounded order.  Thus
$G_n\in{\rm SSGP}(1)$ by Theorem~\ref{torsionSSGP}, and therefore 
$(\ZZ^{(\omega)}, \sT)\in{\rm SSGP}(2)$.  

We show that $(\ZZ^{(\omega)}, \sT)\notin{\rm SSGP}(1)$:  
From the definition of a basic 
neighborhood $U_n$ it is clear that any subgroup of $\ZZ^{(\omega)}$ which $U_n$ 
contains must also be a subgroup of $(n+1)!\ZZ^{(\omega)}$.  Thus the condition
$(\ZZ^{(\omega)},\sT)\in{\rm SSGP}$ would imply that each subroup
$(n+1)! \ZZ^{(\omega)}$ is
dense in $ \ZZ^{(\omega)}$.
Thus every 
nonempty open set is dense in $\ZZ^{(\omega)}$, contradicting Theorem~\ref{prodT2}:
$(\ZZ^{(\omega)}, \sT)$ is a Hausdorff topological group.
\end{proof}

Another way to show that the class-theoretic inclusion
${\rm SSGP}(1)\subseteq{\rm SSGP}(2)$ is proper
is to find a topological group $G\notin{\rm SSGP}$ with
a closed normal subgroup $H$ such that $H\in{\rm SSGP}$ and
$G/H\in{\rm SSGP}$ (for the case $n =m= 1$ of
Theorem~\ref{prsrv_prop}(a) then shows $G\in$ SSGP$(2)$). 
Such examples were given in the second author's dissertation \cite{gouldphd}.  
We generalize that construction to show by induction
for arbitrary $n>1$ the existence of topological 
groups which have SSGP$(n)$ but not SSGP$(n-1)$.

The case $n=1$ is demonstrated by any of our nontrivial SSGP examples.  
However, our construction by induction for the case $n>1$ will require  
additional properties, namely, that our example groups be abelian, countable, torsionfree 
and have a group topology defined by a metric.  Let $H$ be the topological group
$\ZZ^*_A$ as in Definition \ref{deirolf-warken_sub} where $\ZZ$
has the discrete topology and $A \subseteq [0,1)$ consists of  
points of the form
$\frac{t}{2^m}$ for $m, t \in \NN$ and $0 \leq t \leq 2^m$.  
It is clear that $H$ is abelian, countable and torsionfree
and is not 
the trivial group.  $H$ also has SSGP$(1)$
(Theorem~\ref{DW_sub}).  
The topology on $H$ is the metric topology given by the
norm $\|h\| := \lambda(Supp(h))$ for $h \in H$, where 
$Supp(h)$ is the support of $h$ as a function on $[0,1)$. (Here the ``norm"
designation follows historical precedent; we use it both
out of respect and for convenience, but we do not require that 
$\|N g\| = |N| \cdot \|g\|$.)

Fix $n>1$ and suppose there is
a countable,  torsionfree abelian group $G_{n-1}$ with a metrc
$\rho$ that defines a group topology on $G_{n-1}$ such that
$G_{n-1} \in $ SSGP$(n-1)$, and
$G_{n-1} \notin $ SSGP$(n-2)$.  Now, define
(algebraically)
$G_{n}:=H \oplus G_{n-1}$; we give
$G_{n}$ a topology which is different from the 
product topology, using a technique borrowed from M.~Ajtai, I.~Havas, and J.~Koml\'{o}s 
\cite{ajtai}.  We create a metric group topology on $G_{n}$ starting with a 
function  $\nu: S \rightarrow \RR^+$, where $S$ is a specified generating set for 
$G_{n}$ which does not contain $1_{G_n}$.  We refer to $\nu$ together with the generating set $S$ as a ``provisional norm" (in 
terms of which a norm on the group will be defined).  For $g \in G_{n}$ we write 
$g = (h, g^\prime)$ with $h \in H$ and $g^\prime \in G_{n-1}$. 

Let $e_{m,t}$ be the element of $H$ which has
value $1$ on the
interval $[\frac{t-1}{2^m}, \frac{t}{2^m})$ and has value $0=0_\ZZ$ elsewhere.
Let 
$U_m := \{g^\prime \in G_{n-1}: \|g^\prime\| \le \frac{1}{2^m}\}$ for $m \in \ZZ$   
and let $g^\prime_{m,t}$ for $t<\omega$ be a list of the 
elements in $U_m$.  In addition, let $r(m,t)$ be an enumeration of the pairs $(m,t)$.  
We define the set $S \subseteq G_{n}$ to be the
set of elements $s$ in the following 
provisional norm assignments, $\nu(s)$:
\newcounter{item}
\begin{list}{(\arabic{item})}
{\usecounter{item}}
\item $\nu(\,(p \cdot e_{m,t}\,,\,0)\,) = \| p \cdot e_{m,t}\| = \frac{1}{2^m}$ \;\;\;\; for 
$m \in \NN_0, \; 0 < |p|<\omega, \; 1 \le t \le 2^{m}$ 
\item $\nu(\,(f_r\,,\,g^\prime_{m,t})\,) = \|g^\prime_{m,t}\| \le \frac{1}{2^m}$ 
\hspace*{1cm} for  $m \in \ZZ$, $ t<\omega$ \\[.2cm]
where \;\; $\displaystyle{ f_r = \sum_{i=1}^{2^{r}}e_{r,2i-1}}$ \; and \;
$r = r(m,t)$.
\end{list}
Notice that (1) gives the same provisional norm to every non-zero element in a 
subgroup of $G_n$,  whereas (2) is
for a linearly independent 
set of elements of $G_n$.\\

Now we define a seminorm $\| \cdot \|$ on $G_n$
in terms of the provisional norm $\nu$.  

\begin{definition}\label{normdef} For $g \in G_n$,
\begin{center}
$\displaystyle{\|g\| := \inf \left\{\sum_{i=1}^N |\alpha_i| \nu(s_i) : g = 
\sum_{i=1}^N \alpha_i s_i , \,\, s_i \in S , \, \alpha_i \in \ZZ, \, N<\omega\right\}.}$
\end{center}
\end{definition}
This defines a seminorm because $S$ generates $G_n$ and because the use of the 
infimum in the definition guarantees that the triangle inequality will be satisfied.  Therefore, 
the neighborhoods of $0$ defined by this seminorm will generate a
(possibly non-Hausdorff) group topology on $G_n$.  
Again, we do not require that a seminorm (or a norm) satisfy $\|mg\| = |m| \cdot \|g\|$ 
because this property is not needed in order to generate a group topology.

Now in Lemma~\ref{Gn_induct} we use the notation and definition just introduced.

\begin{lemma}\label{Gn_induct}
{\rm (a)} $G_n$ is a  torsionfree, countable abelian group;

{\rm (b)} the seminorm on $G_n$ is a norm (resulting in a Hausdorff metric); 

{\rm(c)} $G_n \in{\rm SSGP}(n)$; and
 
{\rm (d)} $G_n \notin{\rm SSGP}(n-1)$.
\end{lemma}

\begin{proof}
(a) is clear.

(b) To show that $\|\cdot\|$ is a norm on $G_n$, we need to show that for $0\neq g \in G_n$ 
we have $\|g\| > 0$.  Let $g = (h, g^\prime)$.
If $g^\prime \ne 0$ then an 
expansion of $(h, g^\prime)$ by elements of $S$ must include elements
as in (2).  For those 
elements, we have $\nu( \, (f_r , \, g^\prime) \, ) = \|g^\prime\|$
so from the triangle inequality 
in $G_{n-1}$ we can conclude that $\|(h, g^\prime)\|_{G_n} \ge \| g^\prime \|_{G_{n-1}}$.  
On the other hand, if $g^\prime = 0$ then there is an expression for $(h, 0)$ in terms of 
elements of $S$ of type $(p \cdot e_{m,t}, \, 0)$ such that
$\sum_{i=1}^N |\alpha_i| \nu(s_i)  = \|h \| =\lambda(supp(h))$
and this value is minimal.  
If, instead, the expansion includes 
elements of type $s = (f_r\,,\, g^\prime_{m,t})$ then  there is a minimal $\nu$-value such 
a term can have.  This is because there is a minimal size,  
$\frac{1}{2^M}$, for an interval on which $h$ is constant.  An expansion of 
$(h, 0)$ by elements of $S$ that includes an element 
$s = (f_r\,,\, g^\prime_{m,t})$ such that $r(m,t) > M$ would 
also have to include $2^{r}$ elements of $S$ of the form $(e_{r, i}, \, 0)$, each with 
coefficient $-1$.  The contribution of these terms to the sum 
$\sum_{i=1}^N |\alpha_i| \nu(s_i)$ is greater than or
equal to $\frac{1}{2}$.  Such an expansion could have no effect on the infimum.  
So if $\|(h,0)\| \ne \|h\|_H$ then
$\|(h, 0)\| \ge \min( \{\|g^\prime_{m, t}\|_{G_n} : r(m,t) < M \})$.  We 
conclude that $\|(h, \, g^\prime)\|$ is bounded away from $0$
except when
$(h,g^\prime) = (0,0)$.\\

(c) We show next that $G_n \in{\rm SSGP} (n)$.  We use the fact
that $H$ has SSGP in its 
subgroup topology. (This is clear
because the provisional norm 
was defined for each element of the form $(h,0)$, assigning to each 
the norm $\|h\|$ inherited from $\ZZ^*$.) It follows that 
$\|(h,0)\| \le \|h\|$ for each $h \in H$, so any
$\epsilon$-neighborhood of $(0,0)$ contains 
a family $\sH$ of subgroups such that $\langle \sH \rangle = H$.
We will show that 
the quotient topology for $G_n/H$ coincides with the original topology for $G_{n-1}$ 
(which also implies that $\overline{H}^{G_n} = H$).
We use the
fact that the quotient map 
is open, (\cite{hri}(5.26)).  As we just pointed out, for each $g^\prime \in G_{n-1}$ 
there is an $h \in H$ such that 
$\|(h,g^\prime)\| = \| g^\prime\|_{G_{n-1}}$.  We conclude that $g^\prime$ is in the 
$\varepsilon$-neighborhood of $0 \in G_{n-1}$ if and only if there is $h \in H$ such that 
$(h,\, g^\prime)$ is in the $\varepsilon$-neighborhood of $(0,0) \in G_n$.  In other words, the 
neighborhoods of $0$ in $G_{n-1}$ coincide with the projections onto $G_n/H$ of the 
neighborhoods of $(0,0)$ in $G_n$.  Thus the topologies of $G_n/H$ and 
$G_{n-1}$ coincide.  (Note, however, that the \textit{subgroup} topology on $G_{n-1}$ does 
\textit{not} coincide with its original topology.)  Since by assumption 
$G_{n-1} \in$ SSGP$(n-1)$, 
it follows from the definition that $G_n \in$ SSGP$(n)$. \\

(d) It remains to show $G_n \notin$ SSGP$(n-1)$.
Suppose the contrary.  Then every 
$\epsilon$-neighborhood $U_\epsilon$ of $(0, 0) \in G_n$ contains a family $\s K_\epsilon$ of subgroups
such that $G_n/ \, \overline{\langle \bigcup \s K_\epsilon \rangle } \in $ 
SSGP$(n -2)$.    Let $G \in \s K_\epsilon$ and $(h, g^\prime) \in G$ with $g^\prime \ne 0$.  
We claim then $\epsilon \ge \frac{1}{4}$.  For $N<\omega$ we 
must have $\| (Nh, Ng^\prime)\| < \epsilon$.  This means that each $(Nh, Ng^\prime)$ has an 
expansion 
\begin{center}
$(Nh, Ng^\prime) = \sum^{M_N}_{i=1} \alpha^{(N)}_i (h_i,0) + 
\sum^{L_N}_{j=1} \beta^{(N)}_i (h^\prime_i, g^\prime_i)$ \;  such that\\[.2cm]
$\sum^{M_N}_{i=1} \eta(\alpha^{(N)}_i) \, \nu( (h_i,0)) \; + \; 
\sum^{L_N}_{j=1} |\beta^{(N)}_i| \, \nu( (h^\prime_i, g^\prime_i)) \;\;\;\; <  \;\;\;\; \epsilon$\\
\end{center} 
where each $(h_i,0), \; (h^\prime_i, g^\prime_i) \in S$ and where $\eta(\alpha^{(N)}_i) := 1$ 
when $\alpha^{(N)}_i \ne 0$ and $\eta(\alpha^{(N)}_i) := 0$
when $\alpha^{(N)}_i = 0$.  

We consider two cases.  

Case 1. In the  expansion above, the coefficients $\beta^{(N)}_i$ are of the form $N \beta^{(1)}_i$ 
for all $N<\omega$.  Then clearly for sufficiently large $N$ we have 
$\| (Nh, Ng^\prime)\| > \frac{1}{4}$ (or, for that matter,
$\| (Nh, Ng^\prime)\| > \epsilon$).
  
Case 2. There is some $N$ where the expansion for $(Nh, N g^\prime)$ is such that 
$\beta^{(N)}_i \ne N \beta^{(1)}_i$ for some value of $i$.  For the $H$-component of the given 
expansion of $(Nh,Ng^\prime)$, we have 
\begin{center}
$N h = \sum^{M}_{i=1} \alpha^{(N)}_i h_i + 
\sum^{L}_{i=1} \beta^{(N)}_i h^\prime_i$
\end{center}
where $M = \max(M_1, M_N)$ and $L = \max(L_1, L_N)$.
Multiplying the specified expansion 
of $(h, g^\prime)$ by the number $N$, we also have 
\begin{center}
$N h = \sum^{M}_{i=1} N \alpha^{(1)}_i h_i + 
\sum^{L}_{i=1} N \beta^{(1)}_i h^\prime_i$.
\end{center}
Equating the two expansions and re-arranging, we can write
\begin{center}
$\sum^{L}_{j=1} ( \beta^{(N)}_j - N \beta^{(1)}_j) h^\prime_j = 
\sum^M_{i = 1} (N \alpha^{(1)}_i - \alpha^{(N)}_i ) h_i$
\end{center}
where, for some index $j$, we have $(\beta^{(N)}_j - N \beta^{(1)}_j) h^\prime_j \ne 0$.  
Since the $h^\prime_j$ are linearly independent, each $h^\prime_j$
that has a nonzero  coefficient in the expression above
must be balanced by terms on the right.
This implies 
that $\sum^M_{i = 1} \eta(N \alpha^{(1)}_i - \alpha^{(N)}_i )  \ge \frac{1}{2}$,
which in turn 
means that either
$\sum^M_{i = 1} \eta( \alpha^{(1)}_i )  \ge \frac{1}{4}$ or 
$\sum^M_{i = 1} \eta( \alpha^{(N)}_i )  \ge \frac{1}{4}$.  We conclude that 
$\|(n, g^\prime)\| \ge \frac{1}{4}$ or $\|(Nn, Ng^\prime)\| \ge \frac{1}{4}$, as 
claimed. Returning to the family $\sK_\epsilon$ of subgroups,
we see that if
$\epsilon < \frac{1}{4}$ then $\sK_\epsilon \subseteq \sP(H)$ so that 
$\overline{\langle \cup \sK_\epsilon \rangle} $ is a closed subgroup of $H$.  
In such cases we have 
$G_n/ \overline{\langle \cup \sK_\epsilon \rangle} \notin$ SSGP$(n-2)$ 
because, by Theorem \ref{prsrv_prop} part (b), 
$G_n/ \overline{\langle \cup \sK_\epsilon \rangle} \in$ SSGP$(n-2)$ would imply 
that $(G_n/ \overline{\langle \cup \sK_\epsilon \rangle})/
(H/ \overline{\langle \cup \sK_\epsilon \rangle}) \in \rm{SSGP}(n-2)$ or, equivalently,
that $G_n/H \simeq G_{n-1} \in$ SSGP$(n-2)$, contrary to assumption. 
\end{proof}

We emphasize the essential content of Lemma~\ref{Gn_induct}.

\begin{theorem}\label{exists_SSGPn_not_n-1}
For $1 \le n<\omega$ there is an abelian topological group $G$ such
that $G \in$ SSGP$(n)$ 
and $G \notin$ SSGP$(n-1)$. 
\end{theorem}

While Theorem~\ref{dont_admit} furnishes a vast supply of well-behaved abelian
groups which admit no SSGP$(n)$ topology,
we have found that an SSGP topology can be constructed 
for many of the standard building blocks of infinite abelian 
groups. We give now verify Theorem~\ref{admit}(c),
that is, we give a construction
of an SSGP topology for groups of the
form
$G:=\oplus_{p_i}\,\ZZ_{p_i}$ (with $(p_i)$ a squence of
distinct primes);
this illustrates the method used throughout the 
second-listed co-author's thesis \cite{gouldphd}.

Using additive notation, write $0=0_G$
and let $\{e_i:\, i=1,2,. . .\}$ be 
the canonical basis for $G$, so that $p_1e_1 = p_2 e_2 =\ldots=p_ie_i\ldots=0$.  
We define a provisional norm $\nu$, much
as in the description preceding 
Lemma \ref{Gn_induct}.  This will
generate a norm $||\cdot||$ via Definition \ref{normdef}
in such a way that in the generated topology
every neighborhood of $0$ 
contains sufficiently many
subgroups to generate a dense 
subgroup of $G$.  Suppose we can show that $G$ 
is Hausdorff and that each $U \in \sN(0)$ 
contains a family of subgroups $\s H$ such that 
$G/\overline{H}$ is torsion of bounded order, where 
$H := \langle \cup \s H \rangle$.  Then if also 
$G/\overline{H}$ has no proper open subgroup, 
we have from Theorem \ref{torsionSSGP} 
that $G/\overline{H} \in$ SSGP,
so that $G \in \rm{SSGP}(2)$.  
Our plan is to choose a norm so that $G/H$, 
and thus $G/\overline{H}$ is actually finite.  
Then if $G$ contains no proper open subgroup, 
it is necessarily the case that $\overline{H} = G$.  
Thus we attempt to define a norm $\| \cdot \|$ so that 
\begin{list}{(\arabic{item})}
{\usecounter{item}}
\item Every neighborhood of $0$ contains a set of subgroups of 
$G$ whose union generates a subgroup $H$ such that $G/H$ 
is finite.
\item $G$ has no proper open subgroups, or equivalently, every 
neighborhood of $0$ generates $G$.
\item $G$ is Hausdorff.
\end{list}
First, define $\nu(m e_n) = \frac{1}{n}$ for
every $m<\omega$ such that
$m \not \equiv 0$ mod $p_n$.  The neighborhood
of $0$ defined by $||g|| < \frac{1}{n}$ will then contain 
subgroups which generate $H := p_1p_2 . . . p_{n-1} G$.  
Then  $G/H$  is finite, as desired.

To satisfy (2) define $\widehat{e}_n := \sum_{i = 1}^n e_i$
for $n < \omega$, and define $\nu(\widehat{e}_n) := \frac{1}{n}$ 
for each $n < \omega$.  
All that then remains (the most difficult piece) is to show 
that $G$ with this topology is Hausdorff.  We will then 
have the following result.

\begin{theorem}\label{cyclic_sum_diff}
Let $G = \bigoplus_{i < \omega} \ZZ_{p_i}$ where $p_1 < p_2 < p_3 < . . . $ are primes.
Let $S = \{m e_n: n<\omega, \, 0 < m < p_n\} \bigcup \{\widehat{e}_n : n \in \NN \} $, with $e_n, \widehat{e}_n$
defined as above.
Let $\nu(m e_n) = \frac{1}{n}$ for $0 < m < p_n$, and let $\nu(\widehat{e}_n) = \frac{1}{n}$.  Then the norm
defined by
\begin{center}
$\displaystyle{||g|| = \inf \left\{\sum_{i=1}^n |\alpha_i| \nu(s_i) : g = \alpha_1 s_1 +. . .+\alpha_n s_n
, \,\, s_i \in S , \, \alpha_i \in \ZZ, \, n<\omega \right\}}$
\end{center}
generates an {\rm SSGP} topology on $G$.
\end{theorem}

\begin{proof} 
As noted above, our construction for the norm $\| \cdot \|$ guarantees that every 
$\epsilon$-neighborhood $U$ of $0$ generates $G$ and also 
contains subgroups whose union generates 
an $H$ such that $G/H$ is finite.  Then, as also noted, if 
$G$ is Hausdorff  we are done.  

Suppose $0\neq g \in G$ and $n$ is the highest nonzero coordinate index for $g$.   We need to show that
$||g||$ is bounded away from $0$.  In fact, we
show that $||g|| \geq \frac{1}{n}$.  For convenience we
extend the domain of $\nu$ to all formal finite sums of elements from $S$ with coefficients from
$\ZZ$:\\  [.5cm]
\hspace*{1cm}For $\varphi = \sum_{i=M}^N  (a_i e_i +  b_i \widehat{e}_i)$, \hspace{1cm}  let
$\nu(\varphi) = \sum_{i=M}^N ( \eta_i + |b_i|) \frac{1}{i}$\\[.5cm]
where each $\eta_i$ is either $0$ or $1$, according as to whether or not 
$a_i \equiv 0$ mod $p_i$ .

In addition, we will assume
\begin{list}{(\arabic{item})}
{\usecounter{item}}
\item $g = \mathrm{val}(\varphi)$, which means that the formal sum $\varphi$ evaluates to $g \in G$;
\item each $e_i$ and each $\widehat{e}_i$ appears at most once in any formal sum; and
\item $0 \leq a_i < p_i$ for each $i$,
\end{list}
item (2) being justified by the fact that we are ultimately interested in the norm, which minimizes
$\nu(\varphi)$.

Let $\sF(M, N)$ be the set of such formal sums
where $M$ is the smallest coordinate index for a
nonzero coefficient $a_M$ or $b_M$ and
where $N$ is the largest such index.
(Here for $b_i$,
``non-zero" indicates that $b_i$ is not a multiple
of $p_1 p_2 . . . p_i$.)

We want to show that $\nu(\varphi) \geq \frac{1}{n}$ where $g = \mathrm{val}(\varphi)$, where
$\varphi = \sum_{i=M}^N  (a_i e_i +  b_i \widehat{e}_i)$ and where either $a_M$ or $b_M$ is nonzero
and either $a_N$ or $b_N$ is nonzero.  In other words, $\varphi \in \sF(M, N)$. This is clear if
$M \leq n$. 

Suppose first that $N = M = n + 1$.    Then, since the $n + 1$ component
of $g$ is $0$ we have that $a_{n + 1} + b_{n + 1} \equiv 0$ mod $p_{n + 1}$.
Both coefficients are $0$ only
if $g=0$, so either both are nonzero or else $a_{n+1} = 0$
and $b_{n+1} = m p_{n + 1}$ for some $m \neq 0$.  In the first case we have
$\nu(\varphi) \geq \frac{2}{n + 1} > \frac{1}{n}$ and in the second case we have
$\nu(\varphi) \geq \frac{p_{n+1}}{n + 1} > 1 \geq \frac{1}{n}$.

Suppose instead that $M = N > n + 1$.  In this case, we know that the $(N - 1)$ component of $g$
is $0$.  In order for that to be true when $g\neq0$
can be written as $\varphi = a_N e_N + b_N \widehat{e}_N$,
it must be the case that $b_N = m p_{N - 1}$ for some $m \neq 0$.  But then we
have $\nu(\varphi) \geq \frac{p_{N - 1}}{N} \geq 1 \geq \frac{1}{n}$.

Finally, we fix $M$ and use induction on $N$.  Assume that we have already shown that
$\nu(\varphi) \geq \frac{1}{n}$ when $\varphi \in \sF(M, Q)$ for $M \leq Q \leq N - 1$, and suppose that
$\varphi \in \sF(M, N)$.  We treat three cases separately.\\
\hspace*{1cm} (a)  Case 1. $|b_{N-1} + b_N| \geq p_{N - 1}$.  Then\\[.5cm]
\hspace*{3cm}$\nu(\varphi) \geq \frac{|b_{N-1}|}{N-1} + \frac{|b_N|}{N} \geq \frac{p_{N-1}}{N} \geq 1 \geq
\frac{1}{n}$.\\[.5cm]
\hspace*{1cm} (b) Case 2. $b_{N-1} + b_N = 0$.  Then we can delete the terms\\[.5cm]
\hspace*{3cm} $b_{N-1} \widehat{e}_{N - 1} + b_N \widehat{e}_N + a_N e_N$\\[.5cm]
without affecting the value of $\varphi$,
and our induction assumption applies.\\[.5cm]
\hspace*{1cm} (c) Case 3. $|b_{N-1} + b_N| <  p_{N - 1}$ and $b_{N-1} + b_N \neq 0$.  Then if 
we let $\varphi^\prime$ be the formal sum obtained from $\varphi$ by deleting 
$a_N e_N + b_n \widehat{e}_N$ and replacing $b_{N-1} \widehat{e}_{N-1}$ with 
$(b_{N-1} + b_N) \widehat{e}_{N-1}$, we have val$(\varphi^\prime) = $ val$(\varphi) = g$, and\\
\begin{center}
$\displaystyle{\nu(\varphi^\prime) - \nu(\varphi) = \frac{|b_{N-1} + b_N| }{N-1} -
\left (\frac{|b_{N-1}|}{N - 1} + \frac{|b_N|}{N} + \frac{1}{N} \right )
\leq \frac{|b_N|}{N(N-1)} - \frac{1}{N}}.$
\end{center}
We see that this difference is negative or zero as long as $|b_N| \leq N - 1$.  Then, since
$\varphi^{\prime} \in \sF(M, N - 1)$, our induction assumption applies.  If, on the contrary,
$|b_N| \geq N$, we already have $\nu(\varphi) \geq 1 \geq \frac{1}{n}$, and we can conclude 
that $G$ is Hausdorff, as desired.
\end{proof}

\section{Concluding Remarks}

\begin{discussion}\label{history}
{\rm
With no pretense to completeness,
we here discuss briefly some of the literature
relating to the development of the class of m.a.p. groups.

(a)~As we indicated earlier, in effect the class m.a.p.
was introduced in 1930 by
von Neumann~\cite{neumann}, who then together with Wigner~\cite{neumannwig} proved that (even in its discrete topology)
the matrix group $SL(2,\CC)$ is an m.a.p. group.

(b)~In the period 1940--1952, several workers showed that certain real topological linear spaces are m.a.p. groups; several examples, with detailed verification, are given by Hewitt and
Ross~\cite{hri}(23.32).

(c)~We have quoted above at length from the 1980 paper of Prodanov~\cite{prod80}, which showed by ``elementary means''
that the group $\bigoplus_\omega\,\ZZ$ admits an m.a.p. topology.

(d)~Ajtai, Havas and Koml\'os~\cite{ajtai}
proved that each group $G$ of the form $\ZZ$, $\ZZ(p^\infty)$,
or $\bigoplus_n\,\ZZ(p_n)$ (with all $p_n\in\PP$ either
identical or distinct) admits a m.a.p. group topology. 

(e)~Protasov~\cite{protasov} and Remus~\cite{remus88} asked whether every infinite abelian group admits an m.a.p. group topology; the question was deftly settled in the negative by Remus~\cite{remus89} with the straightforward
observation that for
distinct $p,q\in\PP$, every group topology on the
infinite group
$G:=\ZZ(p)\times(\ZZ(q))^\kappa$ (with $\kappa\geq\omega$)
has the property that the homomorphism
$x\rightarrow qx$ maps $G$ continuously onto the compact group $\ZZ(p)$.
(See \cite{comf2}(3.J), \cite{comfremusv}(4.6) for additional discussion.)

(f)~Remus~\cite{remus88} showed that every free abelian
group, also every infinite divisible abelian group, admits an m.a.p. topology.

(g) In view of the cited examples
$\ZZ(p)\times(\ZZ(q))^\kappa$ of
Remus~\cite{remus89}, it was natural for
Comfort~\cite{comf2}(3.J.1) to raise the question: Does
every abelian group which is not torsion of bounded order
admit an m.a.p. topology? What about the
countable case?

(h) Motivated by question (g), Gabriyelyan~\cite{gab1}, \cite{gab2} showed that every infinite finitely generated abelian group admits an m.a.p. topology,
indeed the witnessing topology may be chosen complete in the sense that every Cauchy net converges. Gabriyelyan~\cite{gab3} showed further that an abelian torsion group of bounded
order admits an m.a.p. topology if and only if each of its
leading Ulm-Kaplansky invariants is infinite. (The
reader unfamiliar with the Ulm-Kaplansky invariants
might consult \cite{fuchsii}(\S77); those cardinals also play a
significant role
in \cite{comfdik3} in a setting closely related to the present paper.)

(i) Complete and definitive characterizations of those
(not necessarily torsion) abelian groups $G$ which admit an m.a.p. topology
were given recently by Dikranjan and Shakhmatov~\cite{diksha14}. Among them
are these: (1)~$G$ is connected in its Zariski topology;
(2)~$m\in\ZZ\Rightarrow mG=\{0\}$ or $|mG|\geq\omega$;
(3)~the group ${\rm fin}(G)$ is trivial, i.e.,
${\rm fin}(G)=\{0\}$. (The group ${\rm fin}(G)$,
whose study
was initiated in
\cite{diksha10}(4.4) and continued in \cite{comfdik3}(\S2),
may be defined by the relation

${\rm fin}(G)=\langle\bigcup\{mG:m\in\ZZ,|mG|<\omega\}\rangle$).

Detailed subsequent analysis of the theorems and techniques of 
\cite{diksha14} have allowed those authors to answer the
following two questions in the negative; these questions were posed
in \cite{gouldphd} and in a privately circulated pre-publication
copy of the present manuscript.

       (1) {\it Let $G$ be a group with a normal subgroup $K$ for which $K$ and $G/K$ admit topologies $\sU$ and $\sV$ respectively such that $(K,\sU)\in{\rm m.a.p}$ and $(G/K,\sW)\in{\rm m.a.p.}$ Is there then necessarily a group topology $\sT$ on $G$ such that ($K$ is closed in $(G,\sT)$ and) $(K,\sU)=(K,\sT|K)$ and $(G/K,\sU)=(G/K,\sT_q)$ with $\sT_q$ the quotient topology?}
 
       (2) {\it Let $G$ be a group with a normal subgroup $K$ such that both $K$ and $G/K$ admit {\rm m.a.p.} topologies. Must $G$ admit a {\rm m.a.p.} topology?}
}
\end{discussion}

\begin{remark}
{\rm
In the dissertation \cite{gouldphd}, the second-listed co-author
found it convenient to introduce the class of {\it weak} SSGP groups
(briefly, the WSSGP groups), that is,
those topological groups $G=(G,\sT)$ which contain no proper open 
subgroup and have the property that for every $U \in \s N(1_G)$ there is a 
family of subgroups $\s H \subseteq \s P(U)$ such that 
$H = \overline{\langle \cup \s H \rangle}$ is normal in $G$ and $G/H$ is 
torsion of bounded order.  Subsequent analysis 
(as in Theorem~\ref{torsionSSGP} above) along with the
definitions of the classes SSGP$(n)$ has
revealed the class-theoretic inclusions 
SSGP$(1) \subseteq$ WSSGP $\subseteq$ SSGP$(2)$.  
A consequence of Theorem \ref{torsionSSGP} is that the 
Markov-Graev-Remus examples (as in Remark~\ref{MGR})
are not just WSSGP but are, 
in fact, $\rm{SSGP}$.
And  Prodonov's group, discussed above, belongs
to the class of SSGP$(2)$ groups (implying the m.a.p. 
property).  From these facts we conclude that the class of
WSSGP groups contributes little additional useful
information to the present inquiry, and we have chosen to 
suppress its systematic discussion in this paper.
}
\end{remark}

In Theorems~\ref{finitelycogennotSSGP} and \ref{dont_admit} we have identified
several classes of groups which do not admit an SSGP topology.
That suggests the following natural question.

\begin{question}
What are the (abelian) groups which admit an SSGP topology?
\end{question}

Our work also leaves open this intriguing question:

\begin{question}
Does every abelian group which for some $n>1$ admits an SSGP$(n)$ topology also admit an SSGP topology?
\end{question}

{\rm
There is another important and much-studied class of m.a.p. groups: those 
whose every continuous action on a compact space has 
a fixed point, the so-called f.p.c. groups.  (See, for example, 
\cite{glasner}, \cite{pestov3} and \cite{farah_solecki}.)  The study 
of f.p.c. groups, also known as \textit{extremely amenable groups}, 
led to the formulation of a difficult
long-standing open question in abelian topological 
group theory:
\begin{question}
Do the f.p.c. abelian groups constitute a \textit{proper} 
subclass of the m.a.p. abelian groups?
\end{question} 
{\rm This question was raised by 
E. Glasner in 1998 \cite{glasner}.  Even the characterization 
of abelian m.a.p. groups and abelian f.p.c. groups by 
different big set conditions (with the one characterizing 
f.p.c. groups being the stronger) did not settle the question.  
(See \cite{pestov1} and \cite{ellis-keynes} for details.)  Unfortunately, 
the SSGP property has so far not shed light on this question, either.  
It is known \cite{glasner} that there are f.p.c. topologies 
for $\ZZ$, so f.p.c. groups need not have SSGP.}

\end{document}